\documentclass[10pt]{amsart}
\usepackage[cp1251]{inputenc}
\usepackage[british,bulgarian]{babel}
\usepackage{graphicx}
\usepackage[centertags]{amsmath}
\usepackage{amsfonts}
\usepackage{amssymb}
\usepackage{amsthm}
\usepackage{newlfont}
\newtheorem{thm}{Theorem}[section]

\newtheorem{prop}[thm]{Proposition}
\theoremstyle{definition}

\theoremstyle{remark}
\newtheorem{rem}[thm]{Remark}

\newcommand{\tr}{\mathrm{tr}}

\newcommand{\id}{\mathrm{id}}
\begin{document}
\selectlanguage{british}

\title[Lie Groups as Riemannian manifolds with a circulant structure]
 {Two types of Lie Groups as $4$-dimensional Riemannian manifolds with circulant structure}

\author{Iva Dokuzova, Dimitar Razpopov, Mancho Manev}

\begin{abstract}
A $4$-dimensional Riemannian manifold equipped with an endomorphism of the tangent bundle, whose fourth power is the identity, is considered. The matrix of this structure in some basis is circulant and the structure acts as an isometry with respect to the metric. Such manifolds are constructed on $4$-dimensional real Lie groups with Lie algebras of two remarkable types. Some of their geometric characteristics are obtained.
\end{abstract}




\maketitle

\section*{Introduction}
The study of Riemannian manifolds with an almost product structure occupies an important place in differential geometry.
The systematic development of the theory of Riemannian manifold $M$ with a metric $g$ and an almost product structure $P$ is started by K.~Yano in \cite{11}. The classification of almost product manifolds $(M, P, g)$ with respect to the covariant derivative of $P$ is made by A.~M.~Naveira in \cite{nav}.
The manifolds $(M, P, g)$ with zero trace of $P$ are classified with respect to the covariant derivative of $P$ by M.~Staikova and K.~Gribachev in \cite{S-G}.

The Riemannian manifolds equipped with a circulant structure, whose fourth power is the identity, are considered in \cite{G-R}, \cite{5} and \cite{1}. In particular case, such manifolds could be Riemannian almost product manifolds with zero trace of the structure.

In the present paper we consider a $4$-dimensional Riemannian manifold $(M, Q, g)$, where $Q$ is an endomorphism of the tangent bundle $TM$, whose matrix in some basis of the tangent space at any point of $M$ is a circulant one. Moreover, the fourth power of $Q$ is the identity and $Q$ is an isometry with respect to $g$. In addition, we consider the manifold $(M, P, g)$, where $P=Q^{2}$.
Our purpose is to study Lie groups as examples of the investigated manifolds from the so-called main class $\mathcal{W}_{1}$ of Staikova-Gribachev's classification. Moreover, we suppose that the corresponding Lie algebras have special form with respect to G.Mubarakzyanov's classification \cite{mub}.

The paper is organized as follows. In Sect.~\ref{sec:1}, some necessary facts about the considered manifold are recalled.
In Sect.~\ref{sec:2}, Lie groups with Lie algebras of two special classes are constructed as manifolds of the type studied.
Relations between these classes of Lie algebras and the class $\mathcal{W}_{1}$ of $(M, P, g)$ are given. Some geometric characteristics of the constructed manifolds are obtained.

\section{Preliminaries}\label{sec:1}

Let $M$ be a $4$-dimensional Riemannian manifold with a metric $g$.
Let $Q$ be an endo\-mor\-phism in the tangent space $T_{p}M$ at
an arbitrary point $p\in M$
with local coordinates given by the following circulant matrix with respect to some basis
\begin{equation}\label{f4}
    (Q_{i}^{j})=\begin{pmatrix}
      0 & 1 & 0 & 0\\
      0 & 0 & 1 & 0 \\
      0 & 0 & 0 & 1\\
      1 & 0 & 0 & 0\\
    \end{pmatrix}.
\end{equation}
Then, $Q$ has the properties
\begin{equation}\label{q4}
    Q^{4}=\id,\qquad Q^{2}\neq\pm \id.
\end{equation}
We suppose that $Q$ acts as an isometry with respect to the metric $g$, i.e.
\begin{equation}\label{2.1}
    g(Qx, Qy)=g(x, y).
\end{equation}

Here and anywhere in this work $x, y, z, u$ will stand for arbitrary elements of the algebra of the smooth vector fields on $M$ or vectors in the tangent space $T_{p}M$. The Einstein summation convention is used, the range of the summation indices being always $\{1, 2, 3, 4\}$.

We denote the manifold $M$ equipped with the Riemannian metric $g$ and the circulant structure $Q$ by $(M, Q, g)$.

A basis of $T_{p}M$, which has the type $\{x, Qx, Q^{2}x, Q^{3}x\}$, is called a $Q$-basis.
It is known from \cite{1} that there exists an orthonormal $Q$-basis.


Let $\nabla$ be the Riemannian connection of the metric $g$. The curvature tensor $R$ of $\nabla$ is defined by
$R(x, y)z=\nabla_{x}\nabla_{y}z-\nabla_{y}\nabla_{x}z-\nabla_{[x,y]}z.$
Its associated tensor of type $(0, 4)$ is determined by
$R(x, y, z, u)=g(R(x, y)z,u).$
The Ricci tensor $\rho$ and the scalar curvature $\tau$ are defined as usual by
$\rho(y,z)=g^{ij}R(e_{i}, y, z, e_{j})$ and $\tau=g^{ij}\rho(e_{i}, e_{j})$ with respect to an arbitrary basis $\{e_{i}\}$.
If $\{x, y\}$ is a non-degenerate $2$-plane spanned by $x, y \in T_{p}M$, then its sectional curvature is determined by
$k(x,y)=R(x, y, x, y)\{g(x, x)g(y, y)-g^{2}(x, y)\}^{-1}$.

If we denote $P=Q^{2}$, then the conditions \eqref{q4} and \eqref{2.1} imply $P^{2}=\id$, $P\neq \pm \id$, $g(Px, Py)=g(x,y)$ \cite{5}.
Thus, $(M, P, g)$ is a Riemannian manifold with an almost product structure $P$. Moreover, \eqref{f4} implies $\tr P=0$.
For such manifolds, Staikova-Gri\-ba\-chev's classification is valid. It contains three basic classes $\mathcal{W}_i$ ($i=1,2,3$) \cite{S-G}.
This classification was made with respect to the tensor $F$ of type $(0,3)$ and the Lee form $\theta$ defined by
\begin{equation}\label{F}
  F(x,y,z)=g((\nabla_{x}P)y,z),\quad \theta(x)=g^{ij}F(e_{i},e_{j},x),
\end{equation}
where  $g^{ij}$ are the components of the inverse matrix of $g$ with respect to an arbitrary basis $\{e_{i}\}$.
The tensor $F$ has the following properties:
\begin{equation}\label{Fprop}
	F(x,y,z)=F(x,z,y), \qquad F(x,y,z)=-F(x,Py,Pz).
\end{equation}
The only class of these manifolds, which is closed with respect to the usual conformal transformations of the metric, is $\mathcal{W}_1$. Moreover, only the manifolds of this class have $F$ expressed explicitly by $g$ and $P$. Namely, by definition we have
\begin{equation}\label{c1}
\mathcal{W}_{1}:\; F(x,y,z)=\frac{1}{4}\big\{g(x,y)\theta(z)+g(x,z)\theta(y)
-g(x,Py)\theta(Pz)-g(x,Pz)\theta(Py)\big\}.
\end{equation}
The class of Riemannian $P$-manifolds $\mathcal{W}_{0}$ is defined by the condition
$F(x, y, z)=0$ and $\mathcal{W}_{0}$ is included in all classes.

From here on, we consider Riemannian manifolds equipped with structures $P$ and $Q$ satisfying the relation $P=Q^2$.

According to \cite{G-R}, the condition $\nabla Q=0$ for $(M,Q,g)$ implies that $(M,P,g)$ belongs to $\mathcal{W}_{0}$. In \cite{1}, it is proved that any manifold $(M,Q,g)$ with $\nabla Q=0$ has the curvature property
$R(x, y, Qz, Qu)=R(x, y, z, u)$. The converse is not always true.
In this work, we consider a more general condition of the latter equality which is
\begin{equation}\label{R}
  R(Qx, Qy, Qz, Qu)=R(x, y, z, u).
\end{equation}
In the next section, we find necessary and sufficient conditions for
the manifolds considered to satisfy
the property \eqref{R}.

\section{Lie groups with the structure $(Q, g)$}\label{sec:2}

Let $G$ be a $4$-dimensional real connected Lie group and $\mathfrak{g}$ be its Lie algebra with a basis $\{e_{1}, e_{2},e_{3},e_{4}\}$ of left invariant vector fields. We introduce a circulant structure $Q$ and a metric $g$ as follows
\begin{eqnarray}
&Qe_{1}=e_{4},\quad Qe_{2}=e_{1},\quad Qe_{3}=e_{2},\quad Qe_{4}=e_{3};\label{lie} \\
&g(e_{i}, e_{j})= \delta_{ij},\label{g}
\end{eqnarray}
where $\delta_{ij}$ is the Kronecker delta.
Then, the used basis $\{e_i\}$ is an orthonormal $Q$-basis. Obviously, \eqref{f4}--\eqref{2.1} are valid and $(Q, g)$ is a structure of the considered type. We denote the corresponding manifold by $(G, Q, g)$.

We use Mubarakzyanov's classification \cite{mub} of four-dimensional real Lie alge\-bras. This scheme seems to be the most popular (see \cite{Biggs} and the references therein).
We pay attention to two classes $\{\mathfrak{g}_{4,5}\}$ and $\{\mathfrak{g}_{4,6}\}$, which represent indecomposable Lie algebras,
depending on two real parameters $a$ and $b$. Actually, they induce
two families of manifolds whose properties are functions of $a$ and $b$.
In this work, an object of special interest are the curvature properties of these
manifolds.

\subsection{Lie algebra in $\{\mathfrak{g}_{4,5}\}$}

Consider the case when $\mathfrak{g}$ belongs to $\{\mathfrak{g}_{4,5}\}$.
According to the definition of this class, we have (\cite{Biggs}):
\begin{equation}\label{skobki1}
  [e_{1}, e_{4}]=e_{1},\; [e_{2}, e_{4}]=ae_{2},\; [e_{3}, e_{4}]=be_{3},\quad
  -1\leq b \leq a \leq 1,\; ab\neq 0.
\end{equation}

\begin{prop}\label{kt2}
If $\mathfrak{g}$ belongs to $\{\mathfrak{g}_{4, 5}\}$, then the curvature tensor $R$ of $(G, Q, g)$ satisfies the property \eqref{R} if and only if the condition
$a=b=1$ holds.
\end{prop}
\begin{proof}[Proof]
The well-known Koszul formula implies
\begin{equation*}
    2g(\nabla_{e_{i}}e_{j}, e_{k})=g([e_{i}, e_{j}],e_{k})+g([e_{k}, e_{i}],e_{j})+g([e_{k}, e_{j}],e_{i})
\end{equation*}
 and, using \eqref{g} and \eqref{skobki1}, we obtain
\begin{equation*}
\begin{array}{lll}
    \nabla_{e_{1}}e_{1}=-e_{4},\quad
&    \nabla_{e_{1}}e_{4}=e_{1},\quad
&    \nabla_{e_{2}}e_{2}=-ae_{4},\\\nonumber
   \nabla_{e_{2}}e_{4}=ae_{2},\quad
&    \nabla_{e_{3}}e_{4}=be_{3},\quad
&    \nabla_{e_{3}}e_{3}=-be_{4}.
\end{array}
\end{equation*}
Then, we calculate the components $R_{ijks}$ of $R$ with respect to $\{e_i\}$. The non\-zero of them are obtained from the symmetries of $R$ and the following:
\begin{equation}\label{r1}
\begin{array}{lll}
    R_{1212}=a,\quad & R_{1414}=1,\quad & R_{2323}=ab, \\
    R_{3434}=b^{2},\quad & R_{1313}=b,\quad & R_{2424}=a^{2}.
\end{array}
\end{equation}

According to \eqref{lie}, we obtain that the property \eqref{R} is equivalent to the equalities
\begin{equation}\label{R-loc}
\begin{array}{ll}
R_{1212}=R_{3434}=R_{2323}=R_{1414},\quad & R_{1313}=R_{2424},\\
R_{1213}=R_{2324}=R_{1424}=R_{3134},\quad &
R_{1214}=R_{1434}=R_{2123}=R_{3234}, \\
R_{1224}=R_{3123}=R_{3114}=R_{4234}, \quad & R_{1324}=0.
\end{array}
\end{equation}
Due to \eqref{r1}, we get that \eqref{R-loc} is satisfied if and only if $a=b=1$.
Then, the statement holds.
\end{proof}

Under the conditions of the proposition above, \eqref{r1} takes the form
\begin{align}\label{rlamda}
  R_{1212}=R_{1414}=R_{2323}=R_{3434}=R_{1313}=R_{2424}=1.
\end{align}

\begin{prop}\label{W1}
  If $\mathfrak{g}$ belongs to $\{\mathfrak{g}_{4,5}\}$, then
$(G,P,g)$ belongs to $\mathcal{W}_{1}$ if and only if the condition $a=b=1$ holds.
\end{prop}
\begin{proof}[Proof]
Bearing in mind \eqref{F}, \eqref{lie} and \eqref{g},
we get the
components $F_{ijk}$ of $F$
and $\theta_i$ of $\theta$ with respect to the basis $\{e_i\}$.
The nonzero of them are the following
\begin{equation}\label{F45}
\begin{array}{ll}
  F_{112}=F_{121}=-F_{134}=-F_{143}=1, \quad &  F_{222} =-F_{244} =2a, \\
  F_{323}=F_{332}=-F_{341}=-F_{314}=b, \quad & \theta_{2}=2a+b+1.
\end{array}
\end{equation}
By equalities \eqref{F45}, we get that the condition
\eqref{c1} holds if and only if $a=b=1$.
\end{proof}
\begin{rem}
Obviously, if $\mathfrak{g}$ is in $\{\mathfrak{g}_{4,5}\}$, then $(G, P, g)$ does not belong to $\mathcal{W}_{0}$ under any conditions.
\end{rem}

By virtue of Proposition~\ref{kt2} and Proposition~\ref{W1}, we have immediately the following
\begin{prop}
If $\mathfrak{g}$ belongs to $\{\mathfrak{g}_{4, 5}\}$, then the curvature tensor $R$ of $(G, Q, g)$  satisfies \eqref{R} if and only if  $(G,P,g)$ belongs to $\mathcal{W}_{1}$.
\end{prop}

Next we get
\begin{prop}\label{kt3}
If $\mathfrak{g}\in\{\mathfrak{g}_{4, 5}\}$ and $a=b=1$ are valid, then $(G, Q, g)$ is:
\begin{itemize}
\item[(i)]  an Einstein manifold with a negative scalar curvature $\tau=-8$;
\item[(ii)]  of constant sectional curvatures, as the sectional curvatures of the basic 2-planes are $k_{ij}=1$.
\end{itemize}
\end{prop}
\begin{proof}[Proof]
Using \eqref{rlamda}, we compute the nonzero components of $\rho$, the values of $\tau$ and the sectional curvatures $k_{ij}=k(e_{i}, e_{j})$ of the basic 2-planes and they are as follows:
\begin{equation*}
    \rho_{11}=\rho_{22}=\rho_{33}=\rho_{44}=-2,\quad
    \tau=-8,\quad
    k_{ij}=1.
\end{equation*}
Then, the statement holds.
\end{proof}

\subsection{Lie algebra in $\{\mathfrak{g}_{4,6}\}$}

Now, consider the case when $\mathfrak{g}$ belongs to $\{\mathfrak{g}_{4,6}\}$.
According to the definition of this class, we have (\cite{Biggs}):
\begin{equation}\label{skobki2}
  [e_{1}, e_{4}]=ae_{1},\; [e_{2}, e_{4}]=be_{2}-e_{3},\;
  [e_{3}, e_{4}]=e_{2}+be_{3},\quad a\neq 0,\; b\geq 0.
\end{equation}

\begin{prop}\label{kt4}
If $\mathfrak{g}$ belongs to $\{\mathfrak{g}_{4, 6}\}$, then the curvature tensor $R$ of $(G, Q, g)$
 satisfies \eqref{R} if and only if the condition
$a=b$ holds.
\end{prop}
\begin{proof}[Proof]
It is analogous to the proof of Proposition~\ref{kt2}. The corresponding equalities for the components of $\nabla$ and $R$ have the following form
\begin{equation*}
\begin{array}{llll}
    \nabla_{e_{1}}e_{1}=-ae_{4},\quad &
    \nabla_{e_{1}}e_{4}=ae_{1},\quad &
    \nabla_{e_{2}}e_{2}=-be_{4},\quad &
    \nabla_{e_{3}}e_{4}=be_{3}\\
    \nabla_{e_{2}}e_{4}=be_{2},\quad &
    \nabla_{e_{3}}e_{3}=-be_{4},\quad &
    \nabla_{e_{4}}e_{2}=e_{3},\quad &
    \nabla_{e_{4}}e_{3}=-e_{2};
\end{array}
\end{equation*}
\begin{equation}\label{r2}
    R_{1212}=R_{1313}=ab,\quad R_{1414}=a^{2},\quad
R_{2323}=R_{3434}=R_{2424}=b^{2}.
\end{equation}

 Using \eqref{r2}, we get that \eqref{R-loc}, which is the equivalent form of \eqref{R}, is satisfied if and only if the condition $a=b$ is valid.
\end{proof}

Under the conditions of the latter proposition,  \eqref{r2} takes the form
\begin{align}\label{rlamda2}
  R_{1212}=R_{1414}=R_{2323}=R_{3434}=R_{1313}=R_{2424}=a^{2}.
\end{align}

\begin{prop}\label{W12}
  If $\mathfrak{g}$ belongs to $\{\mathfrak{g}_{4,6}\}$, then
$(G,P,g)$ belongs to $\mathcal{W}_{1}$ if and only if the condition $a=b$ is valid.
\end{prop}
\begin{proof}[Proof]
In this case, the nonzero components of $F$ and $\theta$ are determined by the following equalities:
\begin{equation*}
\begin{array}{ll}
  F_{112} =F_{121}=-F_{143}=-F_{134} =a, \quad
  F_{222}=- F_{244}  =2b, \\
  F_{323}=F_{332}=-F_{314}=- F_{341}  =b,\quad \theta_{2}= a+3b.
  \end{array}
\end{equation*}
Then, we obtain that \eqref{c1} is valid if and only if $a=b$ holds.
\end{proof}
\begin{rem}
Obviously, if $\mathfrak{g}$ is in $\{\mathfrak{g}_{4,6}\}$, then
$(G,P,g)$ belongs to $\mathcal{W}_{0}$ if and only if $a=b=0$ holds, which contradicts the condition for $a$ in \eqref{skobki2}.
\end{rem}

Due to Proposition~\ref{kt4} and Proposition~\ref{W12}, we obtain immediately the following
\begin{prop}
If $\mathfrak{g}$ belongs to $\{\mathfrak{g}_{4, 6}\}$, then the curvature tensor $R$ of $(G, Q, g)$ satisfies \eqref{R} if and only if the manifold $(G,P,g)$ belongs to  $\mathcal{W}_{1}$.
\end{prop}

Next, bearing in mind \eqref{rlamda2}, we get
\begin{prop}
If $\mathfrak{g}\in\{\mathfrak{g}_{4, 6}\}$ and $a=b$ are valid, then $(G, Q, g)$ is:
\begin{itemize}
\item[(i)]  an Einstein manifold with a negative scalar curvature $\tau=-8a^2$;

\item[(ii)] of constant sectional curvatures, as the sectional curvatures of the basic 2-planes are $k_{ij}=a^2$.
\end{itemize}
\end{prop}

\section*{Acknowledgments}

This work is partially supported by project FP17-FMI-008 of the Scientific Research Fund, University of Plovdiv Paisii Hilendarski, Bulgaria.

\selectlanguage{british}

\bigskip

\noindent
Iva Dokuzova\footnote{Corresponding author}$^{,a}$, Dimitar Razpopov$^{c}$, Mancho Manev$^{a,b}$\\
dokuzova@uni-plovdiv.bg,\ razpopov@au-plovdiv.bg,\ mmanev@uni-plovdiv.bg\\[4pt]
$^{a}$Department of Algebra and Geometry\\
Faculty of Mathematics and Informatics \\
University of Plovdiv Paisii Hilendarski \\
24 Tzar Asen St,\ 4000 Plovdiv,\ Bulgaria\\[4pt]
$^{b}$Department of Medical Informatics, Biostatistics and E-Learning\\
Faculty of Public Health \\
Medical University of Plovdiv \\
15A Vasil Aprilov Blvd,\ 4002 Plovdiv,\ Bulgaria\\[4pt]
$^{c}$Department of Mathematics and Physics\\
Faculty of Mathematics, Informatics and Economics \\
Agricultural University of Plovdiv \\
12 Mendeleev Blvd,\ 4000 Plovdiv,\ Bulgaria\\

\end{document}